\documentclass[11pt,reqno]{article}

\usepackage[latin1]{inputenc}

\usepackage{pgf}

\usepackage{amsmath,amsfonts,amssymb,amscd,amsthm,graphicx}

\usepackage{color}
\usepackage{float}
\usepackage{sectsty}
\allsectionsfont{\centering}

\usepackage{amsfonts}
\usepackage{latexsym}

\usepackage{amsmath}
\usepackage{amsthm}
\usepackage{makeidx}
\usepackage{latexsym}
\usepackage{graphicx}
\usepackage{amssymb}

\usepackage{hyperref}
\usepackage{float}
\newtheorem{theorem}{Theorem}

\newtheorem{proposition}{Proposition}
\newtheorem{corollary}{Corollary}
\newtheorem{remark}{Remark}
\newtheorem{definition}{Definition}
\theoremstyle{remark}

\usepackage[latin1]{inputenc}

\title{Polynomial sequences on quadratic curves}

\author{Marco Abrate, Stefano Barbero,\\ Umberto Cerruti, Nadir Murru\\
Department of Mathematics, Turin University,\\ Via Carlo Alberto 10, 10122, Italy\\
marco.abrate@unito.it, stefano.barbero@unito.it,\\ umberto.cerruti@unito.it, nadir.murru@unito.it}

\date{}

\begin{document}

\maketitle

\begin{abstract}
In this paper we generalize the study of Matiyasevich on integer points over conics, introducing the more general concept of \emph{radical points}. With this generalization we are able to solve in positive integers some Diophantine equations, relating these solutions  by means of particular linear recurrence sequences. We point out interesting relationships between these sequences and known sequences in OEIS. We finally show connections between these sequences and Chebyshev and Morgan-Voyce polynomials, finding new identities.
\end{abstract}

\section{Introduction}
\indent \indent  Finding sequences of points that lie over conics is an interesting and well-studied topic in mathematics. An important example is the search for approximations of irrational numbers by sequences of rationals, which can be viewed as sequences of points over conics; see, e.g., \cite{Burger} and \cite{bcm2}.\\ 
\indent Many such investigations involve quadratic curves having points whose coordinates are terms of linear recurrence sequences. Matiyasevich \cite{Mat} showed that points $(x,y)$ belonging to the conic
\begin{equation}\label{Matcon}C(k)=\lbrace (x, y)\in \mathbb{R}^2 : x^2-kxy+y^2=1, k\in\mathbb{N},  k\geq 2\rbrace\end{equation}
are integer points if and only if $(x,y)=(s_{n},s_{n+1})$, where $(s_n)_{n=0}^{+\infty}$ is the linear recurrence sequence with characteristic polynomial $t^2-kt+1$, starting with $s_0=0$ and $ s_1=1$. Similar results have been proved by W. L. McDaniel \cite{Daniel} for a class of conics related to Lucas sequences. Melham \cite{Mela} and Kilic et al. \cite{Kilic} studied conics whose integer points are precisely the points of coordinates $(s_n,s_{n+m})$, $(s_{kn},s_{k(n+m)})$ and provided similar results for different linear recurrence sequences of order 2. Horadam \cite{Horadam}  showed that all consecutive terms of the linear recurrence sequence $(w_n)_{n=0}^\infty$, with characteristic polynomial $t^2-pt+q$ and initial conditions $w_0=a, w_1=b$, where $p$, $q$, $a$, $b \in \mathbb{Z}$, satisfy
$$qw_{n}^2+w_{n+1}^2-pw_nw_{n+1}+eq^n=0$$
where $e=pab-qa^2-b^2$.\\ 
\indent Clearly, these studies are strictly related to the solutions of some Diophantine equations of the second degree with two variables. For example, in \cite{Mills} and \cite{Iran}, linear recurrence sequences have been used in order to solve Diophantine equations $x^2-\alpha xy+y^2+ax+ay +1=0$ and $x^2\pm kxy-y^2\pm x=0$, with $\alpha,a,k \in\mathbb{Z}$, or equivalently to determine all the integer points over the conics described by these equations. In \cite{bcm}, authors used linear recurrence sequences in order to evaluate powers of points over the Pell hyperbola and determine the solutions of the Pell equation in an original way.

Some other interesting results, concerning the characterization of integer points over conics with integer coefficients by means of linear recurrence sequences, can be found, e.g., in \cite{Jones}, \cite{Kimb},\cite{Hon}, and \cite{Jones2}. In this paper, we will consider the curve of equation 
\begin{equation}\label{Matgen}x^2-\sqrt{w}xy+y^2=1, \quad w \in \mathbb{N}, \quad w\geq4, \end{equation} 
which generalizes the Matiyasevich one. We extend the characterization of integer points over this conic, when $w$ is not a perfect square, introducing ''algebraic'' points of the kind $(u\sqrt{w},v)$ or $(u,v\sqrt{w})$ for  $u,v \in \mathbb{N}$. We show that these points are related to particular linear recurrence sequences. The characterization is harder than the integer case, as we will point out in section \ref{sec:mat}. Moreover, in section \ref{sec:dioph}, we will solve some Diophantine equations connected to the quadratic curve \eqref{Matgen} providing new results about the integer sequences involved. Finally, in section \ref{sec:pol}, we will highlight connections with Chebyshev and Morgan-Voyce polynomials, finding new identities.

\section{Linear recurrence sequences on generalized Matiyasevich curves} \label{sec:mat}
In this section we study the quadratic curve
\begin{equation}\label{cw}C(w) = \{(x,y)\in\mathbb R : x^2-\sqrt{w}xy+y^2=1\}, \quad w\in \mathbb{N},\quad w\geq 4 \end{equation}
 that we call \emph{generalized Matiyasevich conic}. For a given value of $w$, we define the set $E(w)=L(w) \cup R(w)$ of \emph{radical points}, where
$$L(w)=\{(u\sqrt{w},v)\in C(w): u,v\in\mathbb N, u\sqrt{w}>v\}$$ 
and 
$$R(w)=\{(u,v\sqrt{w})\in C(w): u,v\in\mathbb N, u>v\sqrt{w}\}.$$
In the following we will characterize all radical points of the generalized Matiyasevich conics by means of consecutive terms of the linear recurrence sequence with characteristic polynomial $t^2-\sqrt{w}t+1$ and initial conditions $0,1$.
\begin{definition} \label{abc}
Let $w$ be a real number with $w\geq4$. We consider 
\begin{align*}
(a_n(w))_{n=0}^{+\infty}& =\left( 0,1,\sqrt{w},w-1,(w-2) \sqrt{w},w^2-3 w+1, \dots\right)\\
(b_n(w))_{n=0}^{+\infty}&= \left( 0,1,w-2,w^2-4 w+3,w^3-6 w^2+10 w-4, \dots \right)\\ 
(c_n(w))_{n=0}^{+\infty}&= \left( 1, w-1, w^2-3w+1, w^3-5w^2+6w-1, \dots \right) \end{align*}
linear recurrence sequences that have characteristic polynomials and initial conditions respectively given by
\begin{align*}
f(t) &= t^2-\sqrt{w}t+1, & a_0(w)=0,&\quad a_1(w)=1,\\
g(t) &= t^2-(w-2)t+1, &b_0(w)=0,& \quad b_1(w)=1,\\
g(t)&=t^2-(w-2)t+1, &c_0(w)=1,&\quad c_1(w)=w-1.
\end{align*}
\end{definition}
In the following we will omit the dependence from $w$ when there is no possibility of misunderstanding. 
\begin{proposition}\label{incr}
The sequence $(a_n)_{n=0}^{+\infty}$ is strictly increasing.
\end{proposition}
\begin{proof}
If $w=4$, it is straightforward to observe that $(a_n)_{n=0}^{+\infty}=(n)_{n=0}^{+\infty}$. Now, let us consider $w>4$ and consequently $\sqrt{w}-1>1$. We prove the thesis by induction. For the first terms we have $a_2=\sqrt{w}>a_1=1>a_0=0$. Given any $k\leq n$, we suppose $a_k>a_{k-1}$
and using the recurrence relation $a_{n+1}=\sqrt{w}a_n-a_{n-1}$, we obtain
$$a_{n+1}-a_n=(\sqrt{w}-1)a_n-a_{n-1}>a_n-a_{n-1}>0$$
and the proof is complete.
\end{proof}
Now  let us introduce the matrix $M$ 
$$M = \begin{pmatrix} 0 & 1 \cr -1 & \sqrt{w} \end{pmatrix},$$
with characteristic polynomial $t^2-\sqrt{w}t+1$. The entries of $M^n$  recur with this polynomial (see \cite{CV}) and checking $M$ and $M^2$, we can find that
$$M^n= \begin{pmatrix} -a_{n-1} & a_n \cr -a_n & a_{n+1} \end{pmatrix}.$$
This equality allows us to prove the following proposition.
\begin{proposition}\label{P}
If we consider $w\in\mathbb{N}$ and points $P_n=(a_n,a_{n-1})$, for all integers $n\geq1$, we have
\begin{align*}
1)\quad & P_n\in C(w),\\
2)\quad & P_{2n}=(b_n\sqrt{w},c_{n-1}),\\
3)\quad &  P_{2n+1}=(c_n,b_n\sqrt{w}).
\end{align*}
\end{proposition}
\begin{proof}
\ \\
\begin{enumerate}
\item Since 
$$a_{n}^2-a_{n-1}a_{n+1}=\det (M^n) =[\det(M)]^n= 1,$$
we obtain $P_n\in C(w)$ because
$$1=a_{n}^2-a_{n-1}(\sqrt{w}a_{n}-a_{n-1})=a_{n}^2-\sqrt{w}a_{n}a_{n-1}+a_{n-1}^2.$$

\item It is immediate to prove that the even terms of $(a_n)_{n=0}^{+\infty}$ are multiple of $\sqrt{w}$. Moreover, even and odd terms of $(a_n)_{n=0}^{+\infty}$ recur with the characteristic polynomial $g(t)=t^2-(w-2)t+1$ of the matrix $M^2$ (see \cite{CV}). Thus the sequences $(a_{2n})_{n=0}^{+\infty}$, $(a_{2n+1})_{n=0}^{+\infty}$, $(b_n)_{n=0}^{+\infty}$, and $(c_n)_{n=0}^{+\infty}$ have characteristic polynomial $g(t)$. Furthermore, observing that 
\begin{align*}
a_0=0,&\quad a_1=1,\quad a_2=\sqrt{w}, \quad a_3=w-1,\\
b_0=0,&\quad b_1=1,\quad c_0=1,  \quad \quad c_1=w-1,
\end{align*}
we get
$$a_{2n}=b_n\sqrt{w},\quad a_{2n+1}=c_{n},\quad n\geq0,$$
and clearly
$$P_{2n}=(a_{2n},a_{2n-1})=(b_n\sqrt{w},c_{n-1}),\quad n\geq1.$$
\item From the previous considerations, we immediately find
$$P_{2n+1}=(a_{2n+1},a_{2n})=(c_n,b_n\sqrt{w}),\quad n\geq0.$$
\end{enumerate}
\end{proof}
As a consequence of Propositions \ref{incr} and \ref{P} we have the following inclusions. 
\begin{corollary}Using the above notation, we have
\begin{align*}
1)\quad & \lbrace P_{2n+1}\in C(w): n\geq0 \rbrace \subseteq R(w),\\
2)\quad & \lbrace P_{2n}\in C(w): n\geq1 \rbrace \subseteq L(w),\\
3)\quad & \lbrace P_{n}\in C(w): n\geq1 \rbrace \subseteq E(w).
\end{align*}
\end{corollary}
\begin{corollary}\label{corbc}
The terms of $(b_n)_{n=0}^{+\infty}$ and $(c_n)_{n=0}^{+\infty}$ satisfy 
$$b_{n+1}=c_n-b_n,\quad c_{n+1}= w b_{n+1}-c_n,\quad n\geq0.$$
\end{corollary}
\begin{proof}
Let us observe that the sequences $(b_n)_{n=0}^{+\infty}$ and $(c_n)_{n=0}^{+\infty}$ recur with characteristic polynomial $g(t)$ and
\begin{align*}
& b_1=c_0-b_0=1,  \quad \quad\quad\quad b_2=c_1-b_1=w-2,\\
& c_1=wb_1-c_0=w-1,\quad  c_2=wb_2-c_1=w^2-3w+1.
\end{align*}
Thus, we have 
$$b_{n+1}=c_n-b_n,\quad c_{n+1}= w b_{n+1}-c_n,\quad n\geq0.$$
\end{proof}
\begin{remark}
If $w=k^2$ for some $k\geq2$, then \eqref{Matgen} is the equation of the conic \eqref{Matcon} studied by Matiyasevich. Moreover, in this case the set of the  radical points is 
$$E(w)=E(k^2)=\{(a_{n+1}(k^2),a_n(k^2)): n \in \mathbb{N} \}$$ 
which corresponds to the set of all integer points over this curve.
\end{remark}

In the next theorem we show that this result can be generalized when $w$ is a non-square positive integer.

\begin{theorem}
For every non-square integer $w \geq 4$, the set of radical points belonging to the conic \eqref{Matgen} is 
$$E(w)=\{(a_{n+1}(w),a_n(w))\in C(w): n\in \mathbb{N}\}.$$
\end{theorem}
\begin{proof}
We will write $\{P_n\}$ instead of $\{P_n=(a_n,a_{n-1})\in C(w): n\geq1\}$. We want to show that $E(w)\subseteq\{P_n\}$. We give the following order to the elements of $E(w)$. Given any $(x_1,y_1), (x_2, y_2)\in E(w)$ we have $(x_1,y_1)<(x_2,y_2)$ if and only if $x_1<x_2$. In this way this set is well-ordered, with minimum element $(1,0)$ and we can prove the theorem by induction. Since, $(1,0)=(a_1,a_0)=P_1$ the induction basis is true. Moreover, let us suppose that if $(r,s)\in E(w)$, for $(r,s)<(\bar x,\bar y)$, then $(r,s)\in\{P_n\}$. We will prove that if $(\bar x, \bar y)\in E(w)$ then $(\bar x, \bar y)\in \{P_n\}$. There are two possible cases: $(\bar x, \bar y)\in L(w)$ and $(\bar x, \bar y)\in R(w).$ 

Let us suppose $(\bar x, \bar y)\in L(w)$, i.e., $(\bar x, \bar y)=(u\sqrt{w},v)$ and $u\sqrt{w}>v$. We know that
\begin{equation} \label{44} u^2w+v^2-uvw=1. \end{equation}
If $v=1$, then $u^2=u$ and consequently we have $u=1$, i.e., $(\bar x, \bar y)=(\sqrt{w},1)=(a_2,a_1)=P_2$. If $v>1$, we have $u=\frac{1-v^2}{uw}+v$ and $v>u$ since $1-v^2<0$. From $u\sqrt{w}>v$, it follows that $uv\sqrt{w}>v^2$ and
\begin{equation} \label{46}  \cfrac{v^2}{u\sqrt{w}}<v. \end{equation}
Dividing \eqref{44} by $u\sqrt{w}$ we obtain
$$u\sqrt{w}=v\sqrt{w}+\cfrac{1}{u\sqrt{w}}-\cfrac{v^2}{u\sqrt{w}}$$
and using $\eqref{46}$, we have
$$u\sqrt{w} > v\sqrt{w}+ \frac{1}{u\sqrt{w}}-v$$
and
$$ v>  v\sqrt{w}- u\sqrt{w} + \frac{1}{u\sqrt{w}}>(v-u)\sqrt{w}.$$
Now, we consider the point $Q=(v,(v-u)\sqrt{w})$. We have that $Q\in C(w)$, indeed
$$v^2+w(v-u)^2-wv(v-u)=v^2+wu^2-uvw=1.$$
Since $v>(v-u)\sqrt{w}$ we obtain $Q\in R(w)\subseteq E(w)$. Moreover, we are considering $v<u\sqrt{w}$, i.e., $Q<(\bar x,\bar y)$. By Proposition \ref{P} and the inductive hypothesis we conclude that $Q=(v,(v-u)\sqrt{w})=(c_n,b_n\sqrt{w})$, for some index $n$. Thus, we have 
$$v=c_n,\quad v-u=b_n$$
and by Corollary \ref{corbc} it follows that
$$v=c_n,\quad u=c_n-b_n=b_{n+1}.$$
Finally,
$$(\bar x,\bar y)=(u\sqrt{w},v)=(b_{n+1}\sqrt{w},c_n)=P_{2(n+1)}\in\{P_n\}.$$
Similar arguments hold in the case $(\bar x,\bar y)\in R(w)$.
\end{proof}

\section{Diophantine equations and integer sequences} \label{sec:dioph}
Let us consider the quadratic curve
$$C_2(w) = \{(x,y)\in\mathbb R : (x+y-1)^2=wxy, w\in \mathbb{N},w\geq 4\}.$$
In the next theorem we determine all the positive integer points of $C_2(w)$, i.e., we solve the Diophantine equation
\begin{equation}\label{C2}(x+y-1)^2=wxy, \quad w\in\mathbb{N},\quad w\geq4.\end{equation}

\begin{definition} \label{upol}
Let $w$ be a real number with $w\geq4$. We define $(u_n(w))_{n=0}^{+\infty}$ as the sequence satisfying the recurrence relation
\begin{equation}\label{urec}
\begin{cases}
u_0(w)=0, \quad u_1(w)=1,\cr
u_{n+1}(w)=(w-2)u_{n}(w)-u_{n-1}(w)+2,\quad n\geq2.
\end{cases}
\end{equation}
\end{definition}
In the following we will omit the dependence on $w$, when there is no possibility of misunderstanding.
\begin{theorem}
The point $(x,y)\in C_2(w)$  has positive integer coordinates if and only if $(x,y)=(u_{n+1},u_n)$ or $(x,y)=(u_n, u_{n+1})$, for some natural number $n$.
\end{theorem}
\begin{proof}
Since equation \eqref{C2} is symmetric with respect to $x$ and $y$, we may consider only the case $(x,y)=(u_{n+1},u_n)$. It is straightforward to prove that
\begin{equation}\label{aux} u_n^2-u_{n-1}u_{n+1}-2u_n+1=0,\quad n\geq1, \end{equation}
and consequently if $(x,y)=(u_{n+1},u_n)$, using the recurrence relation \eqref{urec} and the equality \eqref{aux}, we have
\begin{align*}
(x+y-1)^2-w x y&=u_n^2+u_{n+1}^2-wu_{n+1}u_n+2u_{n+1}u_n-2u_{n+1}-2u_n+1&\\
&=u_{n}^2+u_{n+1}(u_{n+1}-(w-2)u_n-2)-2u_n+1&\\
&=u_n^2-u_{n-1}u_{n+1}-2u_n+1=0,\end{align*}
i.e., $(x,y)\in C_2(w)$ with positive integer coordinates.

Conversely, let $(x,y)\in C_2(w)$ be a point with positive integer coordinates. 
First of all we observe that $x$ and $y$ must be coprime in order to satisfy the equation defining $C_2(w)$.  Furthermore it follows that
$$x|y^2-2y+1,\quad y|x^2-2x+1,$$
and
$$xz=y^2-2y+1,$$
where $z$ is an integer number. Since $xz \equiv 1 \pmod y$, we have
$$x^2(z^2-2z+1)\equiv 1-2x+x^2\equiv 0 \pmod y$$
and
$$y|z^2-2z+1,\quad z|y^2-2y+1,$$
i.e, $(y,z)\in C_2(w)$ is an integer point. Thus, starting with $x$ and $y$ we can determine a sequence where three consecutive elements satisfy
$$y^2-xz-2y+1=0.$$
This equation corresponds to the relation \eqref{aux}. Thus, it follows that $x,y,z$ are three consecutive elements of $(u_n)_{n=0}^{+\infty}$, i.e., $(x,y,z)=(u_{n+1}, u_n, u_{n-1})$ for a given index $n$. Using the results proved in \cite{Mills}, it is easy to show that this sequence must satisfy the recurrence relation \eqref{urec}.
\end{proof}
In the following theorem we highlight the relationships among the quadratic curves $C(w)$, $C_2(w)$ and $C_3(w)$ defined as
$$ C_3(w)=\{(x,y)\in\mathbb R: (x+y)^2=w(x+1)(y+1), w\in \mathbb{N},w\geq4\}.$$
In this way, we obtain the positive integer solutions of the Diophantine equation
\begin{equation}\label{C3}(x+y)^2=w(x+1)(y+1),\quad w\in \mathbb{N},\quad w\geq4. \end{equation}
\begin{theorem}\label{cc2c3}
Considering positive real numbers, we have
\begin{align*}
1)\quad & (x,y)\in C(w) \quad\text{if and only if}\quad (x^2,y^2)\in C_2(w),\\
2)\quad& (x,y)\in C(w) \quad\text{if and only if}\quad (2x^2-1,2y^2-1)\in C_3(w),\\
3)\quad & (x,y)\in C_2(w) \quad\text{if and only if}\quad (2x-1,2y-1)\in C_3(w).
\end{align*}
\end{theorem}
\begin{proof}
\ \\
\begin{enumerate}
\item If $(x,y)\in C(w)$, then $ x^2+y^2-1=\sqrt{w}xy,$ and squaring both members we have $ (x^2+y^2-1)^2= wx^2y^2 $, i.e., $(x^2,y^2) \in C_2(w)$. The converse is obvious.

\item If $(2x^2-1,2y^2-1)\in C_3(w)$, then substituting in \eqref{C3} a little calculation shows that $(x^2+y^2-1)^2=wx^2y^2$. Taking the square root of both members we find $x^2+y^2-1=\sqrt{w}xy$, i.e., $(x,y)\in C(w)$. The converse is obvious.

\item If $(2x-1,2y-1)\in C_3(w)$, then substituting in \eqref{C3} we obtain $(x+y-1)^2=wxy,$, i.e., $(x,y)\in C_2(w).$ The converse is obvious.
\end{enumerate}
\end{proof}

Now, we need some results about operations between linear recurrence sequences. Specifically, we use the product between linear recurrence sequences; see, e.g., \cite{Zierler}, \cite{CV2}, and \cite{Nied}. We mention the main theorem that we will use.
\begin{theorem}
Let $(p_n)_{n=0}^{+\infty}$ and $(q_n)_{n=0}^{+\infty}$ be linear recurrence sequences with characteristic polynomials $f(t)$ and $g(t)$, whose companion matrices are $A$ and $B$, respectively. The product sequence $(p_nq_n)_{n=0}^{+\infty}$, is a linear recurrence sequence with the same characteristic polynomial of the matrix $A \odot B$, where $\odot$ is the Kronecker product \cite{CV2}.
\end{theorem}
\begin{remark} \label{cps}
Let us consider a linear recurrence sequence $(p_n)_{n=0}^{+\infty}$ of order $m$, with characteristic polynomial $f(t)=t^m-\sum_{h=1}^{m}f_ht^{m-h}$ and initial conditions $p_0,...,p_{m-1}$. As a consequence of the previous theorem, the sequence $(q_n)_{n=0}^{+\infty}$, satisfying the recurrence
$$q_m=\sum_{h=1}^m{f_h q_{m-h}}+k, \quad k\in \mathbb{R},$$
 is a linear recurrence sequence with order $m+1$, characteristic polynomial $(t-1)f(t)$ and initial conditions $p_0,...,p_{m-1},p_m+k$.
\end{remark}
From the last remark, we find that sequence $(u_n)_{n=0}^{+\infty}$, introduced in Definition \ref{upol}, is a linear recurrence sequence of degree 3 with characteristic polynomial
$$t^3-(w-1)t^2+(w-1)t-1=(t-1)(t^2-(w-2)t+1)$$
and initial conditions $0,1,w$.  Moreover, we can observe that sequence $(a_n^2)_{n=0}^{+\infty}$, where $(a_n)_{n=0}^{+\infty}$ is the sequence introduced in Definition \ref{abc}, has characteristic polynomial $$t^3-(w-1)t^2+(w-1)t-1$$ and initial conditions $0,1,w$, i.e., 
\begin{equation}\label{unan}(u_n)_{n=0}^{+\infty}=(a_n^2)_{n=0}^{+\infty}.\end{equation}
\begin{remark}
Considering two positive integers $\alpha$ and $\beta$, if the point $(\sqrt{\alpha},\sqrt{\beta})\in C(w)$ is a radical point, then $\alpha$ or $\beta$ are perfect squares. Indeed, if $(\sqrt{\alpha},\sqrt{\beta})\in C(w)$ then $(\alpha,\beta)\in C_2(w)$. Thus, by Theorem 2, there exists an index $n$ such that $(\alpha,\beta)=(u_{n+1},u_n)=(a_{n+1}^2,a_n^2)$ and $(\sqrt{\alpha},\sqrt{\beta})=(a_{n+1},a_n)\in E(w)$.
\end{remark}
The sequence $(u_n(w))_{n=0}^{+\infty}$ is related  to the Chebyshev polynomials of the first kind and we will highlight this connection in the next theorem. We recall that the Chebyshev polynomials of the first kind $T_n(x)$ are defined in various ways (see, e.g., \cite{Riv} and \cite{MH}). Here we define $T_n(x)$ as the $n$-th element of a linear recurrence sequence.
\begin{definition}
The Chebyshev polynomials of the first kind are the terms of the linear recurrence sequence of polynomials $(T_n(x))_{n=0}^{+\infty}$ with characteristic polynomial
$$t^2-2xt+1$$
and initial conditions
$$T_0(x)=1,\quad T_1(x)=x.$$
\end{definition}

\begin{theorem}\label{uT}
For all real numbers $w\geq5$ we have
$$u_n(w)=a_{n}^2(w)=\cfrac{2\left(T_n(\frac{w-2}{2})-1\right)}{w-4},\quad n\geq0.$$
\end{theorem}
\begin{proof}
The sequence $((x-1)u_n(2x+2)+1)_{n=0}^{+\infty}$ has the same characteristic polynomial of $(u_n(2x+2))_{n=0}^{+\infty}$, i.e., 
$$t^3-(2x+1)t^2+(2x+1)t-1=(t-1)(t^2-2xt+1).$$
Thus, $(T_{n}(x))_{n=0}^{+\infty}$ can be considered as a linear recurrence sequence of degree 3 with same characteristic polynomial of $((x-1)u_n(2x+2)+1)_{n=0}^{+\infty}$. Checking that 
$$T_n(x)=(x-1)u_n(2x+2)+1, \quad \text{for}\quad n=0,1,2$$
we have $$(T_n(x))_{n=0}^{+\infty}=((x-1)u_n(2x+2)+1)_{n=0}^{+\infty}.$$
Finally, posing  $x=\frac{w-2}{2}$ and using \eqref{unan}, we obtain
$$u_n(w)=a_{n}^2(w)=\cfrac{2\left(T_n(\frac{w-2}{2})-1\right)}{w-4},\quad n\geq0.$$
\end{proof}

From the previous theorem, we know that consecutive terms of  $$(u_n(w))_{n=0}^{+\infty}= \left(\cfrac{2\left(T_n(\frac{w-2}{2})-1\right)}{w-4}\right)_{n=0}^{+\infty}, \quad w\in\mathbb{N}, \quad w\geq5,$$
are solutions of the Diophantine equations \eqref{C2} and \eqref{C3}. In Proposition \ref{incr} we have studied the case $w=4$, which leads to the sequence $(u_n(4))_{n=0}^{+\infty}=(a_{n}^2(4))_{n=0}^{+\infty}=(n^2)_{n=0}^{+\infty}$.
Furthermore different values of $w$ give rise to many known integer sequences. For all the integer values of $w$ ranging from $4$ to $20$, the sequences are listed in OEIS \cite{oeis}. For example, when $w=5$, we get the sequence of alternate Lucas numbers minus 2 and, when $w=9$, we get the sequence of the squared Fibonacci numbers with even index $$\cfrac{2\left(T_n(\frac{7}{2})-1\right)}{5}=F_{2n}^2, \quad  n\geq0,$$ see sequences A004146 and A049684 in OEIS \cite{oeis}, respectively.
We point out that all these sequences satisfy the recurrence relation \eqref{urec}. The odd terms of the sequences $(u_n(w))_{n=0}^{+\infty}$ are squares and the even terms are squares multiplied by $w$. Indeed, by \eqref{unan} and Proposition \ref{P} we have that 
$$u_{2n}=a_{2n}^2=wb_n^2,\quad u_{2n+1}=a_{2n+1}^2=c_n^2.$$
These equalities clearly show interesting relations between the sequence $(u_n(w))_{n=0}^{+\infty}$ and the sequences $(b_n(w))_{n=0}^{+\infty}$,  $(c_n(w))_{n=0}^{+\infty}$, which are known sequences in OEIS for various values of $w$. We give a little list of these relationships between sequences in Table \ref{table:lines6}, leaving to the reader the pleasure to investigate what happens for other values of $w$.

\begin{table}[hp]
\small
\caption{Relationship between some integer sequences.}
\begin{tabular}{|c|c|c|c|}
\hline 
$w$ & $(u_n(w))_{n=0}^{+\infty}=(a_n^2(w))_{n=0}^{+\infty}$ & $(b_n(w))_{n=0}^{+\infty}$ & $(c_n(w))_{n=0}^{+\infty}$ \cr \hline \hline
4 & A000290=$(n^2)$ & A001477=$(n)$ & A005408=$(2n+1)$  \cr
5 & A004146=Alternate Lucas numbers - 2 & A001906=$(F_{2n})$ & A002878=$(L_{2n+1})$ \cr
6 & A092184 & A001353 & A001834 \cr
7 & A054493 (shifted by one) & A004254 & A030221 \cr
8 & A001108 & A001109 & A002315 \cr
9 & A049684=$F_{2n}^2$ & A004187 & A033890=$F_{4n+2}$ \cr
10 & A095004 (shifted by one) & A001090 & A057080 \cr
11 & A098296 & A018913 & A057081 \cr \hline
\end{tabular}
\label{table:lines6}
\normalsize
\end{table}

In the next section we will highlight further properties connecting previous sequences with Chebyshev polynomials of the second kind and Morgan-Voyce polynomials.

\section{Connections with Chebyshev polynomials of the second kind, Morgan-Voyce polynomials, and new identities} \label{sec:pol}
In the previous section, we proved that the sequence $(u_n(w))_{n=0}^{+\infty}$ can be expressed in terms of the Chebyshev polynomials of the first kind. We showed that these polynomials can be used in order to solve the Diophantine equation \eqref{C2} and determine positive integer points over the conics $C(w)$, $C_2(w)$ and $C_3(w)$.
Now, we find new interesting relations among the previous sequences and some classes of well-known polynomials: the Chebyshev polynomials of the second kind and the Morgan-Voyce polynomials. These polynomials are defined in various ways. We recall their definitions by using linear recurrence sequences.
\begin{definition}
The Chebyshev polynomials of the second kind are the terms of the linear recurrence sequence of polynomials $(U_n(x))_{n=0}^{+\infty}$ with characteristic polynomial
$$t^2-2xt+1$$
and initial conditions
$$U_0(x)=1,\quad U_1(x)=2x.$$
\end{definition}
\begin{definition}\label{chebs}
The Chebyshev polynomials $S_n(x)=U_n\left(\frac{x}{2}\right)$, for all $n\geq0$, are the terms of the linear recurrence sequence of polynomials $(S_n(x))_{n=0}^{+\infty}$ with characteristic polynomial
$$t^2-xt+1$$
and initial conditions
$$S_0(x)=1,\quad S_1(x)=x.$$
\end{definition}
\begin{definition}\label{MV}
The Morgan-Voyce polynomials are the terms of the linear recurrence sequences of polynomials $(f_n(x))_{n=0}^{+\infty}$ and $(g_n(x))_{n=0}^{+\infty}$ with characteristic polynomial
$$t^2-(x+2)t+1$$
and initial conditions
$$f_0(x)=g_0(x)=1,\quad f_1(x)=1+x, \quad g_1(x)=2+x.$$
For properties of these polynomials see, e.g., \cite{Swamy} and \cite{Riley}.\end{definition}
We summarize our results in the next propositions.
\begin{proposition}\label{abcs}
Considering the sequences $(a_n(w))_{n=0}^{+\infty}$, $(b_n(w))_{n=0}^{+\infty}$, $(c_n(w))_{n=0}^{+\infty}$ introduced in Definition \ref{abc} and the Chebyshev polynomials of the second kind introduced in Definition \ref{chebs}, we have
\begin{align*}
1)\quad  a_n(w)&=S_{n-1}(\sqrt{w}),\\
2)\quad  b_n(w)&=S_{n-1}(w-2),\\
3)\quad  c_{n}(w)&=S_{n}(w-2)+S_{n-1}(w-2),
\end{align*}
for $n\geq1$.
\end{proposition}
\begin{proof}
\ \\
\begin{enumerate}
\item Since the sequences $(S_n(\sqrt{w}))_{n=0}^{+\infty}$ and $(a_n(w))_{n=0}^{+\infty}$ have the same characteristic polynomial $t^2-\sqrt{w}t+1$, $a_0(w)=0$, $a_1(w)=1=S_0(\sqrt{w})$, and $a_2(w)=\sqrt{w}=S_1(\sqrt{w})$, for all $n\geq 1$ the sequence $(a_n(w))_{n=0}^{+\infty}$ corresponds to the left shift of the sequence $(S_n(\sqrt{w}))_{n=0}^{+\infty}$.

\item A similar argument shows that for $n\geq1$ the sequence $(b_n(w))_{n=0}^{+\infty}$ is the left shift of the sequence $(S_n(w-2))_{n=0}^{+\infty}$ because they have the same characteristic polynomial $t^2-(w-2)t+1$ and shifted initial conditions.

\item Considering the sequence $(c_n(w))_{n=0}^{+\infty}$ we have from Corollary \ref{corbc} the relation $c_n(w)=b_n(w)+b_{n+1}(w)=S_{n}(w-2)+S_{n-1}(w-2)$.
\end{enumerate}
\end{proof}
\begin{proposition}\label{MVbc}
Considering the sequences $(a_n(w))_{n=0}^{+\infty}$, $(b_n(w))_{n=0}^{+\infty}$, $(c_n(w))_{n=0}^{+\infty}$ introduced in Definition \ref{abc} and the Morgan-Voyce polynomials of the second kind introduced in Definition \ref{MV}, we have
\begin{align*}
1)\quad &  b_n(w)=-g_{n-1}(-w),\quad c_n(w)=f_n(-w), \quad \text{n even},\\
2)\quad & b_n(w)=g_{n-1}(-w),\quad c_n(w)=-f_n(-w),\quad \text{n odd},\\
\end{align*}
for $n\geq 1$.
\end{proposition}
\begin{proof}
We prove the proposition by induction. We have  
$$b_2(w)=w-2=-(2-w)=-g_1(-w) \quad c_2(w)=w^2-3w+1=f_2(-w)$$ and 
$$b_1(w)=1=g_0(-w)\quad c_1(w)=w-1=-(1-w)=-f_1(-w).$$
Thus, the inductive basis is true.
Now let us suppose that the relations hold for all indexes $k\leq n$.
If $n$ is  even, we have that 
$$b_{n+1}(w)=(w-2)b_n(w)-b_{n-1}(w), \quad c_{n+1}(w)=(w-2)c_n(w)-c_{n-1}(w).$$
Observing that $n-1$ is odd, using the inductive hypothesis we obtain
\small
$$b_{n+1}(w)=(w-2)(-g_{n-1}(-w))-g_{n-2}(-w)=(2-w)g_{n-1}(-w)-g_{n-2}(-w)=g_{n}(-w)$$
\normalsize
and
\small
$$c_{n+1}(w)=(w-2)f_n(-w)+f_{n-1}(-w)=-[(2-w)f_n(-w)-f_{n-1}(-w)]=-f_{n+1}(-w)$$
\normalsize
by means of the recurrence relations for the  Morgan-Voyce polynomials.
Since analogous considerations are valid when $n$ is odd, the proof is complete.
\end{proof}
The previous results give a straightforward way to obtain a new proof of some known relations involving Chebyshev and Morgan-Voyce polynomials (see, e.g., \cite{Witula}).
\begin{proposition}
Given any $n\geq1$ and $x\geq2$, we have
\begin{align*}
1)\quad &  S_{2n}(x)=(-1)^nf_{n-1}(-x^2),\\
2)\quad & S_{2n-1}(x)=(-1)^{n-1}xg_{n-1}(-x^2),\\
3)\quad & S_{n}^2(x)-xS_{n}(x)S_{n-1}(x)+S^2_{n-1}(x)=1.
\end{align*}
\end{proposition}
\begin{proof}
By Proposition \ref{P}, Proposition \ref{abcs} and Proposition \ref{MVbc} we have 
$$S_{2n}(\sqrt{w})=a_{2n+1}(w)=c_n(w)=(-1)^nf_{n-1}(-w),$$
$$S_{2n-1}(\sqrt{w})=a_{2n}(w)=\sqrt{w}b_n(w)=(-1)^n\sqrt{w}g_{n-1}(-w).$$
Thus, with the substitution $x=\sqrt{w}$ the proof of 1) and 2) is straightforward.
Finally, from Proposition \ref{P}  we obtain
$$a_{n+1}^2(w)-\sqrt{w}a_{n+1}(w)a_{n}(w)+a_{n}^2(w)=1$$
and by Proposition \ref{abcs} we have 
$$S_{n}^2(\sqrt{w})-\sqrt{w}S_{n}(\sqrt{w})S_{n-1}(\sqrt{w})+S_{n-1}^2(\sqrt{w})=1.$$
If we pose $x=\sqrt{w}$ the proof of 3) easily follows.
\end{proof}

\end{document}